\documentclass[11pt,a4paper]{article}
\usepackage[latin1]{inputenc}
\usepackage{amsthm, amssymb, amscd}
\usepackage{amsmath, amsfonts}
\usepackage{hyperref}
\hypersetup{colorlinks=true, citecolor=blue, urlcolor=blue,linkcolor=blue}

\newtheorem{theorem}{Theorem}
\newtheorem{lemma}{Lemma}
\newtheorem{proposition}{Proposition}
\newtheorem{corollary}[theorem]{Corollary}
\newtheorem{remark}{Remark}

\newcommand{\var}{\mbox{\rm Var}}
\newcommand{\cov}{\text{cov}}

\newcommand{\Esp}{\mathbb E}

\def\R{\mathbb R}

\def\indicator{\mathbf{1}}

\author{
Federico Dalmao\thanks{
Departamento de Matem\'{a}tica y Estad\'{i}stica del Litoral, 
Universidad de la Rep\'{u}blica, A.P. 50000, Salto, Uruguay. 
E-mail: fdalmao@unorte.edu.uy.}
\thanks{Mathematical Research Unit, Universit\'e du Luxembourg, 
6 rue Coudenhove-Kalergi, L-1359, Luxembourg, Luxembourg.}}
\title{CLT for the zeros of Kostlan Shub Smale random polynomials}
\begin{document}
\maketitle
\begin{abstract}
In this paper we find the asymptotic main term of the variance of the number of roots 
of Kostlan-Shub-Smale random polynomials and prove a central limit theorem for the number of roots as the degree goes to infinity.\\
\begin{center}{\bf Resum\'e}\end{center}
Dans ce papier nous trouvons le terme asymptotique dominant de la variance du nombre de racines r\'{e}els des polyn\^{o}mes al\'{e}atoires de Kostlan-Shub-Smale et d\'{e}montrons un 
th\'{e}or\`{e}me de la limite centrale pour ce nombre de racines.
\end{abstract}

Consider the Kostlan-Shub-Smale  (KSS for short) ensemble of random polynomials:
\begin{equation*}%\label{eq:kss}
X_d(x):=\sum^d_{n=0}a^{(d)}_n\,x^n;\quad x\in\R,
\end{equation*}
where $d$ is the degree of the polynomial and the coefficients $(a^{(d)}_n)$ are independent centered Gaussian random variables 
whose variances are the binomial coefficients, more precisely $\var(a^{(d)}_n)=\binom{d}{n}$. 

Denote by $N_d$ the number of real roots of $X_d$, that is
\begin{equation*}
N_d:=\#\{x\in\R:X_d(x)=0\}.
\end{equation*}
It is well known that $\Esp(N_{d})=\sqrt{d}$ \cite{ke,ss}. 
The aim of this paper is to prove the following result.
\begin{theorem}\label{teo}
The variance of the number of real roots $N_d$ of KSS random polynomials verifies 
$$
\lim_{d\to\infty}\frac{\var(N_d)}{\sqrt{d}}=\sigma^2,
$$
with $0<\sigma^2<\infty$ given in Proposition \ref{prop} ($\sigma^2\approx 0.57\dots$).
Furthermore, $N_d$ verifies the CLT
\begin{equation*}
\frac{N_d-\sqrt{d}}{d^{1/4}}\mathop{\Rightarrow}\limits_{d\to\infty} N(0;\sigma^2).
\end{equation*}
\end{theorem}
The number of roots of random polynomials has been under the attention 
of physicists and mathematicians for a long time. 
The first results for particular choices of the coefficients  
are due to Bloch and Polya \cite{bp} in 1932.  
After many successive improvements and generalisations, in 1974 Maslova \cite{ma} stated the CLT for the number of zeros 
for polynomials with i.i.d. centered coefficients with finite variance. 
For related results see the review by Bharucha-Reid and Sambandham \cite{bs} or the Introduction in \cite{adl} and references therein. 

The study of Kostlan-Shub-Smale ($m\times m$ systems) of polynomials started 
in the early nineties  
by Kostlan \cite{ke}, Bogomolny, Bohigas and Lob{\oe}uf \cite{bbl} and Shub and Smale \cite{ss}. 
The mean number of roots \cite{ss,ke}, some asymptotics as $m\to\infty$ for the variance 
\cite{aw2} and 
for the probability of not having any zeros on intervals \cite{sm} are known.
See also the review by Kostlan \cite{k2} and references therein.

We restrict our attention to the case $m=1$. The mean number of real roots 
is $\sqrt{d}$. 
This fact shows a remarkable difference with the polynomials with i.i.d. centered coefficients 
which asymptotic mean number of roots is $2\log(d)/\pi$. 

%Afterwards, by probabilistic methods, some results about the 
%asymptotic variance as the number of variables and equations go to infinity were achieved \cite{aw2}.
%Finally, some results on the probability of having exactly $k$ roots in an interval were obtained \cite{}. 

%
Our tools are the Rice formulas for the (factorial) moments of the number of roots \cite{aw,cl}; 
Kratz-Le\'on's version of the chaotic expansions for the number of zeros \cite{kl97}, Kratz-Le\'on method \cite{kl01} 
and the Fourth Moment Theorem \cite{pt}. 
This method has been applied to trigonometric polynomials 
by %Aza\"{i}s and Le\'on \cite{al} and 
Aza\"is, Dalmao and Le\'on \cite{adl}.

The paper is organised as follows. 
Section 1 contains some preliminaries and sets the problem in a more convenient way. 
Section 2 deals with the asymptotic behaviour of the variance of $N_d$. 
In Section 3 the asymptotic normality of the standardised $N_d$ is obtained.
Section 4 contains the proofs of some auxiliary lemmas. 

\section{Preliminaries}
By the binomial theorem, the covariance of the KSS polynomials is
\begin{equation*}
\cov(X_d(x),X_d(y))=\sum^d_{n=0}\binom{d}{n}x^ny^n=(1+xy)^d.
\end{equation*}
This fact suggest to %use the trick of 
homogenise the polynomials, that is, 
to introduce an auxiliary variable $x_0$ and to consider the polynomials
\begin{equation*}
X^0_d(x_0,x)=\sum^d_{n=0}a_nx^nx_0^{d-n};
\end{equation*}
with $x_0,\,x\in\R$. The polynomial $X^0_d$ is homogeneous, 
that is $X^0_d(\lambda x_0,\lambda x)=\lambda^dX^0_d(x_0,x)$ 
for any $\lambda;\,x_0,\,x\in\R$. 
Therefore, we can think of $X^0_d$ as acting on the unit circumference $S^1$. 
The covariance in this case gives
\begin{equation*}
\cov(X^0_d(x_0,x),X^0_d(y_0,y))=\sum^d_{n=0}\binom{d}{n}(xy)^n(x_0y_0)^{d-n}=(xy+x_0y_0)^d
=\left\langle(x_0,x),(y_0,y)\right\rangle^d,
\end{equation*}
where $\left\langle,\right\rangle$ stands for the usual inner product in $\R^2$.

Furthermore, denoting by $N_Y(I)$ the  number of zeros of the process $Y$ on the set $I$; 
it is easy to see that 
$$
2N_d(\R)=N_{X^0_d}(S^1).
$$

Now, it is convenient to write $X^0_d$ again as a polynomial in one variable by 
identifying $(x_0,x)\in S^1$ with the pair $(\sin(t),\cos(t))$ for some real $t$:
\begin{equation*}
Y_d(t):=\sum^d_{n=0}a_n\cos^n(t)\sin^{d-n}(t),
\end{equation*}
with real $t$. 
It is easy to see that
\begin{equation*}
N_d(\R)=N_{Y_d}([0,\pi])\textrm{ almost surely.}%=\frac{1}{2}N^{Y_d}([-\pi,\pi]).
\end{equation*}
In fact, $x$ is a real root of $X_d$ if and only if 
%$t=\arctan(x)$ is a root $Y_d$. 
the radial projections of $(1,x)$ onto $S^1$, once identified with a point $(\sin(t),\cos(t))$, 
are roots of $Y_d$; one of these projections correspond to $t\in[0,\pi]$.

Direct computations show that $Y_d$ is a centered stationary Gaussian process  
and that, for $s,t\in[0,\pi]$, its covariance function is given by  
\begin{equation*}
\cov(Y_d(s),Y_d(t))=\left\langle(\cos(s),\sin(s)),(\cos(t),\sin(t))\right\rangle^d
=\cos^d(t-s).
\end{equation*}
Furthermore,
\begin{equation*}
\cov(Y_d(s),Y^\prime_d(t))=-d\cos^{d-1}(t-s)\sin(t-s),
\end{equation*}
and 
\begin{equation*}
\cov(Y^\prime_d(s),Y^\prime_d(t))=-d(d-1)\cos^{d-2}(t-s)\sin^2(t-s)+d\cos^d(t-s).
\end{equation*}
In particular, $\var(Y_d(t))=1$ and $\var(Y^\prime_d(t))=d$ for all $t$. \\

\noindent{\bf Expectation of the number of roots of $Y_d$:}\\
Since the process $Y_d$ is stationary and for each fixed $t$ the random variables $Y_d(t)$ and $Y^\prime_d(t)$ 
are independent centered Gaussian with variances $1$ and $d$ respectively, 
using Rice formula \cite{aw}, we have
\begin{equation*}
\Esp\left(N_{Y_d}[0,\pi]\right)=\pi\cdot\Esp|Y^{\prime}_d(0)|\cdot p_{Y_d(0)}(0)
=\pi\cdot\sqrt{\frac{2}{\pi}}\sqrt{d}\cdot\frac{1}{\sqrt{2\pi}}=\sqrt{d}.
\end{equation*}

\noindent{\bf Time scale and covariance limit:}\\
The next step is to scale the time in order to get a limit behaviour for the covariances. 
It is convenient to use the unit speed parametrisation, 
so we define
\begin{equation*}%\label{eq:defZ}
Z_d(t):=Y_d\left(\frac{t}{\sqrt{d}}\right).
\end{equation*}
The number of real roots of $X_d$ coincides almost surely with that 
of $Z_d$ in $[0,\sqrt{d}\pi]$,  
that is
\begin{equation*}
N_{d}(\R)=N_{Z_d}([0,\sqrt{d}\pi])\textrm{ almost surely}.
\end{equation*}

From now on, we restrict the process $Z_d$ 
to the interval $[0,\sqrt{d}\pi]$. 
Since $Y_d$ is stationary, so is $Z_d$. 
Let us denote by $r_d:[-\sqrt{d}\pi,\sqrt{d}\pi]\to\R$ the covariance function of $Z_d$, 
that is, $r_d(t)=\cov\left(Z_d(0),Z_d(t)\right)$. 
It follows that
\begin{align}\label{eq:ryderivadas}
r_d(t)&=\cos^d\left(\frac{t}{\sqrt{d}}\right),\notag \\
r^{\prime}_d(t)&=-\sqrt{d}\cos^{d-1}\left(\frac{t}{\sqrt{d}}\right)\sin\left(\frac{t}{\sqrt{d}}\right),\\
r^{\prime\prime}_d(t)&=(d-1)\cos^{d-2}\left(\frac{t}{\sqrt{d}}\right)\sin^2\left(\frac{t}{\sqrt{d}}\right)
-\cos^{d}\left(\frac{t}{\sqrt{d}}\right).\notag
\end{align}

\begin{remark}
Note that $r_d$ is an even function 
and for $t\in[0,\sqrt{d}\pi/2]$ we have $r_d(\sqrt{d}\pi-t)=(-1)^dr_d(t)$. 
This will imply that it suffices to deal with $r_d$ restricted to the interval $[0,\sqrt{d}\pi/2]$, 
as we shall see in Lemma \ref{lemma:Rice}.
\end{remark}

\section{Asymptotic variance}
We need to prepare some preliminaries.
\begin{lemma}\label{lemma:Rice}
We have
\begin{equation}\label{eq:Rice}
\Esp(N_d(N_d-1))
=\frac{2\sqrt{d}}{\pi}\int^{\sqrt{d}\pi/2}_{0}
g_d(t)\left[\sqrt{1-\rho^2_d(t)}+\rho_d(t)\arctan\left(\frac{\rho_d(t)}{\sqrt{1-\rho^2_d(t)}}\right)\right]dt+1.
\end{equation}
Here  
$g_d(t)=2\pi p_d(t)v_d(t)$, being 
$p_d(t)=p_{Z_d(0),Z_d(t)}(0,0)$ the joint density of $Z_d(0),Z_d(t)$ evaluated at $(0,0)$;  
$v_d(t)$ the conditional variance of $Z^\prime_d(0)$ 
(and of $Z^\prime_d(t)$) conditioned to $Z_d(0)=Z_d(t)=0$ 
and $\rho_d(t)$ the conditional correlation between the derivatives 
$Z^\prime_d(0)$ and of $Z^\prime_d(t)$
conditioned to $Z_d(0)=Z_d(t)=0$. 
We have
\begin{align*}
v_d(t)&=1-\frac{d\cos^{2d-2}\left(\frac{t}{\sqrt{d}}\right)
\sin^2\left(\frac{t}{\sqrt{d}}\right)}{1-\cos^{2d}\left(\frac{t}{\sqrt{d}}\right)},\\
p_d(t)&=\frac{1}{2\pi\sqrt{1-\cos^{2d}\left(\frac{t}{\sqrt{d}}\right)}},\\
\rho_d(t)&=\cos^{d-2}\left(\frac{t}{\sqrt{d}}\right)
\frac{1-d\sin^2\left(\frac{t}{\sqrt{d}}\right)-\cos^{2d}\left(\frac{t}{\sqrt{d}}\right)}
{1-\cos^{2d}\left(\frac{t}{\sqrt{d}}\right)
-d\cos^{2d-2}\left(\frac{t}{\sqrt{d}}\right)\sin^2\left(\frac{t}{\sqrt{d}}\right)}.
\end{align*}
\end{lemma}

Now, we pass to the asymptotic variance of $N_d=N_{Z_d}([0,\sqrt{d}\pi])$. 
We need the following asymptotics and bounds.

\begin{lemma}\label{lemma:limcov}
For each fixed $t\in\R$, 
we have
\begin{equation*}
\cos^d\left(\frac{t}{\sqrt{d}}\right)\mathop{\longrightarrow}\limits_{d\to\infty}e^{-t^2/2}.
%d\sin^2\left(\frac{t}{\sqrt{d}}\right)&\mathop{\longrightarrow}\limits_{d\to\infty}t^2.
\end{equation*}
Besides, 
these convergences are uniform in compacts. 
% en 0, c conv unif, pero precisamos eso? no alcanza con la cota unif?
Furthermore, 
for $0<a<1$ we have the 
following upper bounds
$$
\cos^d\left(\frac{t}{\sqrt{d}}\right)\leq
\begin{cases}
e^{-\alpha t^2/2};&\textrm{ if }0\leq t<a\sqrt{d},\\
\cos^d(a);&\textrm { if }a\sqrt{d}\leq t\leq \pi\sqrt{d}/2.
\end{cases}
$$ 
with $\alpha=1-a^2/3\in(2/3,1)$. 
\end{lemma}
\begin{remark}
It is worth to say that this limit covariance defines a centered stationary Gaussian process 
$X$ on $[0,\infty)$. 
The asymptotic behaviour of the number of real roots of $X_d$ is intimately related to 
the asymptotic behaviour of the number of roots of $X$ in increasing intervals. 
Similar situations occur in \cite{al} (where this fact is indeed used explicitly to obtain the CLT) and \cite{adl}. 
The fact that the limit process $X$ has Gaussian covariance function, and thus Gaussian spectral density, is remarkable.

Nevertheless, we do not need this fact in the sequel.
\end{remark}

\begin{lemma}\label{lemma:varZd}
For fixed $t$, 
\begin{align*}
g_d(t)=2\pi\cdot v_d(t)\cdot p_d(t)&\mathop{\longrightarrow}\limits_{d\to\infty}
\frac{1-(1+t^2)e^{-t^2}}{(1-e^{-t^2})^{3/2}}=:g(t);\\
\rho_d(t)&\mathop{\longrightarrow}\limits_{d\to\infty} 
e^{-t^2/2}\frac{1-t^2-e^{-t^2}}{1-e^{-t^2}-t^2e^{-t^2}}=:\rho(t).
\end{align*}
Besides, 
$0\leq g(t)<1$, $|\rho(t)|\leq 1$ and 
$g(t)\to_{t\to0} 0$. 
%and $|\rho_d(t)|\sim_{t\to\infty}t^2e^{-t^2/2}$. 
Furthermore, 
there exists an integrable upper bound for the r.h.s. of Equation \ref{eq:Rice}.
\end{lemma}

\begin{proposition}[limit variance for $N_d$]\label{prop}
As $d\to\infty$ we have
\begin{equation*} 
\frac{\var(N_d)}{\sqrt{d}}\mathop{\to}\limits_{d\to\infty} \sigma^2
:=\frac{2}{\pi}\int^{\infty}_{0}\left(
g(t)\left[\sqrt{1-\rho^2(t)}+\rho(t)\arctan\left(\frac{\rho(t)}{\sqrt{1-\rho^2(t)}}\right)\right]-1\right)dt+1,
\end{equation*}
where 
$g$ and $\rho$ are defined in Lemma \ref{lemma:varZd}. 
Furthermore, $\sigma^2<\infty$.
\end{proposition}
\begin{remark}
Using Mehler formula we can write also
\begin{equation*}
\sigma^2=\int^{\infty}_{0}
\sum^\infty_{\ell=0}\frac{a^2_{2\ell}}{(2\ell)!}\rho^{2\ell}(t)
(g(t)-\delta_{0\ell})dt,
\end{equation*}
being $a_{2\ell}=2(-1)^{\ell+1}/(\sqrt{2\pi}2^\ell\ell!(2\ell-1))$ 
and $\delta$ Kronecker's delta.
\end{remark}
\begin{proof}
Recall that 
$$\var(N_d)=\Esp(N_d(N_d-1))-(\Esp(N))^2+\Esp(N).$$ 
From Lemma \ref{lemma:Rice} the normalized second factorial moment is
\begin{equation}\label{eq:Rice3}
\frac{\Esp(N_d(N_d-1))}{\sqrt{d}}=
\frac{2}{\pi}\int^{\sqrt{d}\pi/2}_{0}
g_d(t)\sqrt{1-\rho^2_d(t)}dt
+\frac{2}{\pi}\int^{\sqrt{d}\pi/2}_{0}
g_d(t)\rho_d(t)\arctan\left(\frac{\rho_d(t)}{\sqrt{1-\rho^2_d(t)}}\right)dt.
\end{equation}

Let us look at the second integral in the r.h.s. of Equation \eqref{eq:Rice3}. 
Let $a\sqrt{d}\leq t\leq \sqrt{d}\pi/2$.  
We go back in our scaling: $s\mapsto t/\sqrt{d}$; 
so $s\in[a,\pi/2]$. 
By the proof of Lemma \ref{lemma:varZd} we have 
$|\rho_d(t)|\leq\cos^{d-2}(a)$ 
(see Equation \eqref{eq:cotacos} below for the details). 
Also by Lemma \ref{lemma:varZd} $g_d(t)$ is bounded by constant. 
Hence,
\begin{equation*}
\frac{2}{\pi}\int^{\sqrt{d}\pi/2}_{a\sqrt{\pi}}g_d(t)
|\rho_d(t)|\arctan\left(\frac{\rho_d(t)}{\sqrt{1-\rho^2_d(t)}}\right)dt
\leq\int^{\sqrt{d}\pi/2}_{a\sqrt{d}}\cos^{d-2}(a)dt\to_{d\to\infty}0.
\end{equation*}
Let $0\leq t<a\sqrt{d}$. 
Lemma \ref{lemma:varZd} gives the point-wise limit and the domination in order to obtain
\begin{equation*}
\frac{2}{\pi}\int^{a\sqrt{\pi}}_0g_d(t)
\rho_d(t)\arctan\left(\frac{\rho_d(t)}{\sqrt{1-\rho^2_d(t)}}\right)dt
\to_d
\frac{2}{\pi}\int^{\infty}_{0}g(t)
\rho(t)\arctan\left(\frac{\rho(t)}{\sqrt{1-\rho^2(t)}}\right)dt
\end{equation*}
In particular, this integral is finite.

The first integral in the r.h.s. of Equation \eqref{eq:Rice3} 
cancel at infinity with $(\Esp(N_d))^2$. 
In fact, we can write
\begin{equation*}
(\Esp N_d)^2=d=\frac{2}{\pi}\sqrt{d}\int^{\sqrt{d}\pi/2}_0dt.
\end{equation*}
Hence, the first integral in the r.h.s. of Equation \eqref{eq:Rice3} 
minus $(\Esp(N_d))^2$ gives
\begin{equation}\label{eq:conjugado}
\frac{2}{\pi}\int^{\pi\sqrt{d}/2}_{0}
\left[g_d(t)\sqrt{1-\rho^2_d(t)}-1\right]dt
=\frac{2}{\pi}\int^{\pi\sqrt{d}/2}_{0}\left[
\frac{v_d(t)\sqrt{1-\rho^2_d(t)}}{\sqrt{1-r^2_d(t)}}-1\right]dt.
\end{equation}
From Lemma \ref{lemma:varZd} it follows that the integrand 
of the l.h.s of Equation \eqref{eq:conjugado} tends to 
$g(s)-1$. 

In order to obtain a domination, 
by standard manipulation, it follows that the important point 
is to bound the difference 
$v^2_d(t)(1-\rho^2_d(t))-1+r^2_d(t)=(v^2_d(t)-1+r^2_d(t))-(v^2_d(t)\rho^2_d(t))$ 
in the numerator. 
The second term is easily bounded. 
For the first one, 
we have
\begin{equation*}
v^2_d(t)-1+r^2_d(t)
=\left[1-\frac{d\cos^{2d-2}\left(\frac{t}{\sqrt{d}}\right)\sin^2\left(\frac{t}{\sqrt{d}}\right)}
{1-\cos^{2d}\left(\frac{t}{\sqrt{d}}\right)}\right]^2
-1+\cos^{2d}\left(\frac{t}{\sqrt{d}}\right)
%=\frac{\cos^{4d}\left(\frac{t}{\sqrt{d}}\right)+d^2\cos^{4d-4}\left(\frac{t}{\sqrt{d}}\right)
%\sin^4\left(\frac{t}{\sqrt{d}}\right)-2\cos^{2d}\left(\frac{t}{\sqrt{d}}\right)}
%{\left[1-\cos^{2d}\left(\frac{t}{\sqrt{d}}\right)\right]^2}\\
%\frac{-2d\cos^{2d-2}\left(\frac{t}{\sqrt{d}}\right)\sin^2\left(\frac{t}{\sqrt{d}}\right)
%+2d\cos^{4d-2}\left(\frac{t}{\sqrt{d}}\right)
%\sin^2\left(\frac{t}{\sqrt{d}}\right)}
%{\left[1-\cos^{2d}\left(\frac{t}{\sqrt{d}}\right)\right]^2}\\
%\frac{+3\cos^{2d}\left(\frac{t}{\sqrt{d}}\right)-3\cos^{4d}\left(\frac{t}{\sqrt{d}}\right)
%+\cos^{6d}\left(\frac{t}{\sqrt{d}}\right)}
%{\left[1-\cos^{2d}\left(\frac{t}{\sqrt{d}}\right)\right]^2}.
\end{equation*}
After expanding the squares and cancelling the ones, 
we divide again into the cases $0\leq t<a\sqrt{d}$ and $a\sqrt{d}\leq t\leq \sqrt{d}\pi/2$. 
In the latter, the bound $\cos(t/\sqrt{d})<\cos(a)$ suffices to obtain that  
the integral  tends to $0$. 
In the former, an uniform integrable upper bound for this difference 
follows easily by the triangle inequality  
and using the bounds $d^{j}\sin^{2j}(t/\sqrt{d})\leq t^{2j}$; $\cos(t/\sqrt{d})<e^{-\alpha t^2/2}$ and $\cos^{-1}(t/\sqrt{d})<\cos^{-1}(a)$. 

The denominator is handled using the bound in Lemma \ref{lemma:limcov}. 
This gives the domination, so we can pass to the limit inside the integral. 
Hence, the first integral in the r.h.s. of Equation \eqref{eq:Rice3} 
minus $(\Esp(N_d))^2$ tend, as $d\to\infty$, to
\begin{equation*}
\frac{2}{\pi}\int^{\infty}_{0}(g(t)-1)dt.
\end{equation*}
In particular, this integral is finite, thus, so is $\sigma^2$.

Finally, 
let us say that the convergence at $0$ of the integral in Equation \eqref{eq:conjugado} 
follows from the proof of Lemma \ref{lemma:varZd}. 
The result follows.
\end{proof}

\section{CLT}
\begin{proposition}\label{prop2}
The normalised number of zeros of KSS polynomials 
converge in distribution towards a centered Gaussian random variable with variance $\sigma^2$.
\end{proposition}
The idea of the proof is the following: 
to write the normalised number of zeros as a chaotic series \cite{kl97} and then 
to use the Fourth Moment Theorem \cite{pt} combined with Kratz-Le\'on method \cite{kl01} in order to obtain the asymptotic normality. 

More precisely, 
we take the It\^{o}-Wiener expansion of the normalised number of zeros \cite{kl97}. 
Then, 
by Kratz-Le\'{o}n method \cite{kl01}, 
the finiteness of the variance of $N_d$ allows to truncate the expansion  
and to derive its asymptotic normality from that of the sum of the first, say $Q$, terms. 
Finally, 
the Fourth Moment Theorem \cite{pt} gives a criterion to prove the asymptotic normality 
of the finite partial sums of the expansion. 

\begin{proof}
We apply Kratz-Le\'on expansion \cite{kl97} to the processes $Z_d$ on the interval $[0,\sqrt{d}\pi]$.  
Hence
\begin{equation*}
\frac{N_d-\Esp(N_d)}{d^{1/4}}=
\sum^{\infty}_{q=2}I_{q,d},
%=\sum^{\infty}_{q=2}\sum^{[q/2]}_{\ell=0}\frac{b_{q-2\ell}a_{2\ell}}{d^{1/4}}
%\int^{\sqrt{d}\pi/2}_{-\sqrt{d}\pi/2}H_{q-2\ell}\left(Z_d(s)\right)H_{2\ell}\left(Z_d^{\prime}(s)\right)ds,} 
\end{equation*}
where 
\begin{equation}\label{eq:IW}
I_{q,d}=\frac{1}{d^{1/4}}\int^{\sqrt{d}\pi}_{0}f_q(Z_d(t),Z^\prime_d(t))dt,\quad
f_q(x,y)=\sum^{[q/2]}_{\ell=0}b_{q-2\ell} a_{2\ell}H_{q-2\ell}(x)H_{2\ell}(y),
\end{equation}
with $a_{2\ell}=2(-1)^{\ell+1}/(\sqrt{2\pi}2^\ell\ell!(2\ell-1))$, $b_k=\frac{1}{k!}\varphi(0)H_k(0)$. 
Note that we can delete the term corresponding to $q=1$ since $H_1(0)=0$; 
this is why we restrict our attention to zeros. 

We can express $I_{q,d}$ as multiple stochastic integrals w.r.t $B$. 
In the first place, using a standard B.m. $B$ we can write $Z_d(t)=\int_\R h_d(t,\lambda)dB(\lambda)$ with
\begin{equation}\label{eq:hd}
h_d(t,\lambda)=\sum^d_{n=0}\sqrt{\binom{d}{n}}\cos^n\left(\frac{t}{\sqrt{d}}\right)
\sin^{d-n}\left(\frac{t}{\sqrt{d}}\right)\indicator_{[n,n+1]}(\lambda).
\end{equation}
Then, from Equation \eqref{eq:hd}, 
using the properties of the chaos and the stochastic Fubini theorem, see \cite[Remark 2]{adl}, we have  
$I_q=I^B_q(g_q({\boldsymbol\lambda}_q))$ with
\begin{equation*}
g_q({\boldsymbol\lambda}_q)
=\frac{1}{d^{1/4}}\int^{\sqrt{d}\pi}_{0}\sum^{\left\lfloor q/2\right\rfloor}_{j=0}
b_{q-2j}a_{2j}(h^{\otimes q-2j}_d(s,{\boldsymbol\lambda}_{q-2j})\otimes h_d^{\prime\,
\otimes 2j}(s,{\boldsymbol\lambda}_{2j}))ds;
\end{equation*}
where ${\boldsymbol\lambda}_k\in\R^k$ and $\otimes$ stands for tensorial product.

Now, to get the asymptotic normality of the standardised zeros, 
by Kratz-Le\'on method and the Fourth Moment Theorem, 
it suffices to prove that the contractions $g_q\otimes_k g_q({\boldsymbol\lambda}_{2q-2k})$ tend to $0$ 
in $L^2$ as $d\to\infty$ for $q\geq 2$ and $k=1,\cdots,q-1$ and $\lambda_{2q-2k}\in\R^{2q-2k}$. 

Let ${\boldsymbol z}_k=(z_1,\dots,z_k)$ and 
${\boldsymbol\lambda}_{2q-2k}={\boldsymbol\lambda}_{q-k}\otimes{\boldsymbol\lambda}^\prime_{q-k}$. 
The contractions are defined \cite{pt} 
as
\begin{equation*}
g_{q}\otimes_k g_{q}({\boldsymbol\lambda}_{2q-2k})
=\int_{\R^k}g_{q}({\boldsymbol z}_k,{\boldsymbol\lambda}_{q-k})
g_{q}({\boldsymbol z}_k,{\boldsymbol\lambda}^\prime_{q-k})d{\boldsymbol z}_k.
\end{equation*}

Actually, 
by the properties of stochastic integrals, we have 
$I^B_q(g_q({\boldsymbol\lambda}_q))=I^B_q(Sym(g_q({\boldsymbol\lambda}_q)))$ 
being 
\begin{equation*}
Sym(g_q)({\boldsymbol\lambda}_q)=
\frac{1}{q!}\sum_{\sigma\in S_q}g_q({\boldsymbol\lambda_{\sigma}}),
\end{equation*}
being $S_q$ the group of permutations of $\{1,\dots,q\}$ and 
${\boldsymbol\lambda}_{\sigma}=(\lambda_{\sigma(1)},\dots,\lambda_{\sigma(q)})$. 
So we compute contractions for $Sym(g_q)({\boldsymbol\lambda_{q}})$ instead of 
$g_q({\boldsymbol\lambda_{q}})$. 

Writing down the norm of the contractions is quite tedious, but the basic 
fact is that the isometric property of stochastic integrals implies that 
$h_d^{\otimes p}(s)\otimes_kh_d^{\otimes p}(t)=r^k_d(t-s)h_d^{\otimes p-k}(s)\otimes h_d^{\otimes p-k}(t)$. Similarly, when the identified variable in the contraction  
involves the derivatives of $h$ the result involves the derivatives of $r_d$. 
Taking this into account, it follows that
\begin{multline*}
\|Sym(g_{q})\otimes_k Sym(g_{q})({\boldsymbol\lambda}_{2q-2k})\|^2_2= % sin el absolute value
\frac{1}{d}
\iiiint_{[0,\sqrt{d}\pi]^4}
\sum_{0\leq{\boldsymbol j}\leq[q/2]}c_{\boldsymbol j}\frac{1}{q!}\sum_{\sigma\in S_q}\\
\prod^2_{i=0}(r^{(i)}_d(t-s))^{\alpha_i} (r^{(i)}_d(t^\prime-s^\prime))^{\beta_i}
(r^{(i)}_d(s-s^\prime))^{\gamma_i}
(r^{(i)}_d(t-t^\prime))^{\delta_i}\;ds\;dt\;ds^\prime\;dt^\prime;
\end{multline*}
where ${\boldsymbol j}=(j_1,j_2,j_3,j_4)$, vector inequalities are understood component-wise; 
$c_{\boldsymbol j}=\prod^4_{i=1}a_{2j_i}\,b_{q-2j_i}$;  
$\alpha_i=\alpha_i(\sigma,{\boldsymbol j})$, $\beta_i=\beta_i(\sigma,{\boldsymbol j})$, 
$\gamma_i=\gamma_i(\sigma,{\boldsymbol j})$ and $\delta_i=\delta_i(\sigma,{\boldsymbol j})$; 
$\sum^4_{i=1}\alpha_i=\sum^4_{i=1}\beta_i=k$ and 
$\sum^4_{i=1}\gamma_i=\sum^4_{i=1}\delta_i=q-k$. 
Actually, there are some constrains for $\alpha,\beta,\gamma,\delta$ with respect to 
${\boldsymbol j}$, (namely $\alpha_1\leq (q-2j_1)\wedge(q-2j_2),$ 
$\alpha_2\leq(q-2j_1)\wedge 2j_2+(q-2j_2)\wedge 2j_1$, etc), 
but they are irrelevant for our purposes.

We bound the covariances by their absolute value.  
Since $\var(Z_d(t))=\var(Z^\prime_d(t))=1$, by Cauchy-Schwarz, 
each factor $|r^{(i)}_d(\cdot)|\leq 1$. 
Furthermore, 
since $k\geq 1$ and $q-k\geq 1$, we can bound from above the product of each group of  factors (i.e.: with the same argument) by one of them.

Hence, for some $i_1,i_2,i_3,i_4\in\{0,1,2\}$ we have
\begin{multline*}
\|Sym(g_{q})\otimes_k Sym(g_{q})({\boldsymbol\lambda}_d)\|^2_2\leq\\
\frac{C}{d}
\iiiint_{[0,\sqrt{d}\pi]^4}
|r^{(i_1)}_d(t-s)\,r^{(i_2)}_d(t^\prime-s^\prime)\,r^{(i_3)}_d(s-s^\prime)
\,r^{(i_4)}_d(t-t^\prime)|\;ds\;dt\;ds^\prime\;dt^\prime,
\end{multline*}
where $C$ is a meaningless constant. 
Now, we make the change of variables: $(x,y,u,t^\prime)\mapsto(t-s,t^\prime-s^\prime,s-s^\prime,t^\prime)$ and enlarge the domain of integration in order to have a rectangular one. 
Thus
\begin{equation*}
\|Sym(g_{q})\otimes_k Sym(g_{q})({\boldsymbol\lambda}_d)\|^2_2\leq
\frac{C}{d}
\int^{\sqrt{d}\pi}_{0}dt^\prime
\int^{\sqrt{d}\pi}_{-\sqrt{d}\pi}\;|r^{(i_1)}_d(x)|dx
\int^{\sqrt{d}\pi}_{-\sqrt{d}\pi}\;|r^{(i_2)}_d(y)|dy
\int^{\sqrt{d}\pi}_{-\sqrt{d}\pi}
|r^{(i_3)}_d(u)|\;du.
\end{equation*}
Let us look at the three inner integrals. 
Note that since $r_d$ is even, so is the absolute value of its derivatives, so it suffices to integrate on $[0,\sqrt{d}\pi]$. 
Besides, 
since for $t\in[0,\sqrt{d}\pi/2]$ we have $r_d(\sqrt{d}\pi-x)=(-1)^d r_d(x)$, 
it follows that $|r^{(i)}_d(\sqrt{d}\pi-x)|=|r^{(i)}_d(x)|,i\in\{0,1,2\}$. 
Then, we can further restrict the domain of integration to $[0,\sqrt{d}\pi/2]$. 
The finiteness of the integral then follows from Equation \eqref{eq:ryderivadas}, by bounding the covariance 
by a polynomial (of degree at most $2$) times $\cos^d(\cdot/\sqrt{d})$ and then using Lemma \ref{lemma:limcov}. 
Hence, the contractions tend to $0$. 

The result follows.
\end{proof}
\begin{corollary}\label{coro}
The asymptotic variance $\sigma^2$ is strictly positive. 
\end{corollary}
\begin{proof}
From the It\^{o}-Wiener expansion \eqref{eq:IW} it follows that 
$\sigma^2=\sum^\infty_{q=2}\var(I_{q})$ with $\var(I_q)=\lim_d\var(I_{q,d})$. 
Thus, it suffices to prove that $\var(I_2)>0$. 
This is done exactly as in \cite[Eq 10.42-10.43]{aw}. 
\end{proof}

\begin{proof}[Proof of Theorem \ref{teo}]
Put together Propositions \ref{prop} and \ref{prop2} and Corollary \ref{coro}. 
The approximated value for $\sigma^2$ is obtained numerically from the formula 
in Proposition \ref{prop}. 
\end{proof}
It worth to say that this value is confirmed by simulations.

\section{Proofs of the lemmas}
\begin{proof}[Proof of Lemma \ref{lemma:Rice}]
We compute the second factorial moment via Rice formula. 
By Equation (10.7.5) of \cite{cl} we have 
\begin{equation*}
\Esp(N_d(N_d-1))=\frac{2}{\pi^2}
\int^{\sqrt{d}\pi}_{0}(\sqrt{d}\pi-t)g_d(t)\left(\sqrt{1-\rho^2_d(t)}+\rho_d(t)
\arctan\left(\frac{\rho_d(t)}{\sqrt{1-\rho^2(t)}}\right)\right)dt
\end{equation*}
%In the first place, using the change of variables $(\tau,\sigma)=(t-s,t+s)$, we have
%\begin{multline*}
%\Esp(N_d(N_d-1))
%=\int^{\sqrt{d}\pi/2}_{-\sqrt{d}\pi/2}\int^{\sqrt{d}\pi/2}_{-\sqrt{d}\pi/2}g(\tau)dsdt\\
%=\frac{1}{2}\int^{0}_{-\sqrt{d}\pi}g(\tau)d\tau\int^{\sqrt{d}\pi+\tau}_{-\sqrt{d}\pi-\tau}d\sigma
%+\frac{1}{2}\int^{\sqrt{d}\pi}_0g(\tau)d\tau\int^{\sqrt{d}\pi-\tau}_{-\sqrt{d}\pi+\tau}d\sigma\\
%=\frac{1}{2}\int^{0}_{-\sqrt{d}\pi}g(\tau)(2\sqrt{d}\pi+2\tau)d\tau
%+\frac{1}{2}\int^{\sqrt{d}\pi}_0g(\tau)(2\sqrt{d}\pi-2\tau)d\tau\\
%=\frac{1}{2}\int^0_{\sqrt{d}\pi}\underbrace{g(-\tau)}_{=g(\tau)}(2\sqrt{d}\pi-2\tau)(-d\tau)
%+\frac{1}{2}\int^{\sqrt{d}\pi}_0g(\tau)(2\sqrt{d}\pi-2\tau)d\tau\\
%=\int^{\sqrt{d}\pi}_0g(\tau)(2\sqrt{d}\pi-2\tau)d\tau
%=2\int^{\sqrt{d}\pi}_0g(\tau)(\sqrt{d}\pi-\tau)d\tau
%\end{multline*}
Denote 
$$f(t)=g_d(t)\left(\sqrt{1-\rho^2_d(t)}+\rho_d(t)
\arctan\left(\frac{\rho_d(t)}{\sqrt{1-\rho^2(t)}}\right)\right).$$ 
Then, $f(\sqrt{d}\pi-t)=f(t)$. 
This follows from the properties 
$\cos(t)=\cos(-t)=-\cos(\pi-t)$ and $\sin(t)=-\sin(-t)=\sin(\pi-t)$ 
and from the fact that we only use even powers of the cosines and sines in $g_d$ and $\rho^2_d$ 
and that the signs also cancel in the product 
$\rho\arctan(\rho/\sqrt{1-\rho^2})$.

Then, using the change of variables $x=\sqrt{d}\pi-t$ 
in the interval $[\sqrt{d}\pi/2,\sqrt{d}\pi]$, we have
\begin{multline*}
\Esp(N_d(N_d-1))=
2\int^{\sqrt{d}\pi}_0f(t)(\sqrt{d}\pi-t)dt\\
=2\int^{\sqrt{d}\pi/2}_0f(t)(\sqrt{d}\pi-t)dt+2\int^{\sqrt{d}\pi}_{\sqrt{d}\pi/2}f(t)(\sqrt{d}\pi-t)dt\\
=2\int^{\sqrt{d}\pi/2}_0f(t)(\sqrt{d}\pi-t)dt
+2\int^{0}_{\sqrt{d}\pi/2}f(\sqrt{d}\pi-x)(\sqrt{d}\pi-[\sqrt{d}\pi-x])(-dx)\\
=2\int^{\sqrt{d}\pi/2}_0f(t)(\sqrt{d}\pi-t)dt+2\int^{\sqrt{d}\pi/2}_0f(x)x dx
=2\int^{\sqrt{d}\pi/2}_0f(t)\sqrt{d}\pi dt
\end{multline*}
\end{proof}

\begin{proof}[Proof of Lemma \ref{lemma:limcov}]
% By Equation \eqref{eq:ryderivadas}, we have $r_d(t)=\cos^d(t)$. 

We start with the cosine. 

Assume that $0\leq t<a\sqrt{d}$. 
Using Taylor-Lagrange expansion up to the first order 
for the cosine, we can write
$$
\cos^d\left(\frac{t}{\sqrt{d}}\right)=\left(1+\frac{x}{d}\right)^d;
\textrm{ with }x=-\frac{t^2}{2}\left(1-\frac{\sin(t^*)t}{3\sqrt{d}}\right),
$$
where $t^*\in[0,t/\sqrt{d}]\subset[0,a]$. 
 
There exists $c(a)$ such that $|x/d|<c(a)<1$. 
({\footnotesize In fact, if $t<a\sqrt{d}$, then $x/d=-a^2/2+\sin(t^*)a^3/6$.  
Since the summands have different signs, it follows that $-1<-a^2/2\leq x/d\leq a^4/6<1$. 
The claim follows with $c(a)=\max\{a^2/2,a^4/6\}=a^2/2$. })

Again by Taylor-Lagrange expansion, for the logarithm this time, it follows that
$$
\log\left(\cos^d\left(\frac{t}{\sqrt{d}}\right)\right)
=\log\left(1+\frac{x}{d}\right)^d=x+\log^{\prime\prime}(1+x^{**})\frac{x^2}{2d};
$$
with $x^{**}\in[0,x/d]\subset[0,a]$. 
Hence, 
$$
\cos^d\left(\frac{t}{\sqrt{d}}\right)=\left(1+\frac{x}{d}\right)^d=e^xe^{\log^{\prime\prime}(1+x^{**})\frac{x^2}{2d}}.
$$
Note that $\log^{\prime\prime}(1+x^{**})=-(1+x^{**})^{-2}\in(-(1-a)^{-2},-(1+a)^{-2}]$; 
so $\log^{\prime\prime}(1+x^{**})<0$. \\

Now, for fixed $t$ is easy to see that when $d\to\infty$ then $|t|<a\sqrt{d}$; 
$x\to-t^2/2$ and $\log^{\prime\prime}(1+x^{**})x^2/2d\to 0$. 
Thus, $\cos^d\left(\frac{t}{\sqrt{d}}\right)\to e^{-t^2/2}$ as $d\to\infty$ as claimed. 

It is also easy to check the uniformity of the convergence in compacts $[0,\beta]$ from the preceding expression. 
In fact, let $t\in[0,\beta]$, then
\begin{multline*}
\left|\cos^d\left(\frac{t}{\sqrt{d}}\right)-e^{-t^2/2}\right|
=e^{-t^2/2}\left|e^{\sin(t^*)t^3/6d^{3/2}+\log^{\prime\prime}(1+x^{**}(t))x^2(t)/2d}-1\right|\\
\leq e^{-t^2/2}\max\{\left|e^{\sin(\beta^*)\beta^3/6d^{3/2}}-1\right|,
\left|e^{-(1-\beta^2/2d)^{-2}\beta^4/4d}-1\right|\}\to_d 0;
\end{multline*}
where we use the fact that the summands in the exponent have different signs.\\

Now, we turn to the bounds. 
Since $\log^{\prime\prime}(1+x^{**})<0$ we see that $\cos^d(t/\sqrt{d})\leq e^x$. 
Furthermore, 
since $sin(\cdot)\leq\cdot$, we have $0\leq\sin(t^*)t/3\sqrt{d}\leq a^2/3$. 
Hence,
\begin{equation}\label{eq:rdominacion}
\cos^d\left(\frac{t}{\sqrt{d}}\right)\leq e^x=\exp\left\{-\frac{t^2}{2}\left(1-\frac{\sin(t^*)t}{3\sqrt{d}}\right)\right\}
\leq \exp\left\{-\frac{t^2}{2}\left(1-\frac{a^2}{3}\right)\right\}
\leq e^{-\alpha t^2/2},
\end{equation}
as claimed.
%This shows the uniform convergence and 
%This gives us a simple uniform integrable upper bound for the covariance on the region 
%$|t|<a\sqrt{d}$.\\

% \marginpar{precisa algo mas?}
Assume now that $a\sqrt{d}\leq t\leq \sqrt{d}\pi/2$. 
In this case, we have $\cos(t/\sqrt{d})<\cos(a)<1$, hence $\cos^d(t/\sqrt{d})<\cos^d(a)$. 
The result follows.
%The result for the sine follows (in the same way) using its Taylor-Lagrange expansion (recall that $t$ is fixed):
%$$
%d\sin^2\left(\frac{t}{\sqrt{d}}\right)=d\left[\frac{t}{\sqrt{d}}-\frac{\sin(t^*/\sqrt{d})t^2}{2d}\right]^2
%=t^2-\frac{\sin(t^*/\sqrt{d})t^3}{\sqrt{d}}+\frac{\sin^2(t^*/\sqrt{d})t^4}{4d}.
%$$
%The result follows.
\end{proof}

\begin{proof}[Proof of Lemma \ref{lemma:varZd}]
% \marginpar{rellenar}
The limits are direct consequences of Lemma \ref{lemma:limcov}.

Bounds: 
Let us look at the domination in a neighbourhood of $t=0$. 
First, note that $\rho_d(t)\leq 1$ since it is a correlation; so the sum is finite. 
%($\leq\sum a^{2}_{2\ell}/(2\ell)!$).

Consider the factor 
$$2\pi\cdot v_d(t)\cdot p_d(t)=
\frac{1-\cos^{2d}\left(\frac{t}{\sqrt{d}}\right)-d\cos^{2d-2}\left(\frac{t}{\sqrt{d}}\right)\sin^{2}\left(\frac{t}{\sqrt{d}}\right)}
{\left(1-\cos^{2d}\left(\frac{t}{\sqrt{d}}\right)\right)^{3/2}}.$$

By Taylor-Lagrange expansion, as in the proof of Lemma \ref{lemma:limcov}, we can write 
$$1-\cos^{2d}\left(\frac{t}{\sqrt{d}}\right)=1-e^{u}=1-\left(1+u+e^{u^*}\frac{u^2}{2}\right)=-u-e^{u^*}\frac{u^2}{2};$$ 
where 
\begin{align*}
u=2x+\log^{\prime\prime}(1+x^{**})\frac{x^2}{d};\quad x^{**}\in[0,x/d]\subset[0,x];\\
x=-\frac{t^2}{2}+\frac{\sin(t^*)t^3}{6\sqrt{d}};\quad t^*\in[0,t/\sqrt{d}].
\end{align*}

First, we bound $x$; 
since $t$ lies on a neighbourhood of $0$ (assume, for simplicity, that $t<1$), the signature of the sinus is the same as that of $t$. Hence,
$$
-\frac{t^2}{2}\leq x=-\frac{t^2}{2}+\frac{\sin(t^*)t^3}{6\sqrt{d}}
\leq-\frac{t^2}{2}+\frac{t^4}{6d}\leq-\frac{t^2}{2}+\frac{t^4}{6}.
$$
%Hence,
%$$
%-\frac{t^2}{2}\leq x\leq-\frac{t^2}{2}+\frac{t^4}{6}
%$$

Now, we pass to the second term in the definition of $u$:  
since $x<0$ and thus $x^{**}<0$, we have
$$
-\log^{\prime\prime}(1+x^{**})=\frac{1}{(1+x^{**})^2}\geq 1.
$$
Hence,
$$
-\log^{\prime\prime}(1+x^{**})\frac{x^2}{d}\geq\frac{x^2}{d}
\geq\frac{1}{d}\left(-\frac{t^2}{2}+\frac{t^4}{6}\right)^2.
$$
On the other hand, $x^{**}\geq x\geq -t^2/2$, so $1+x^{**}\geq 1-t^2/2$. 
Hence,
$$
-\log^{\prime\prime}(1+x^{**})\frac{x^2}{d}\leq\frac{1}{(1-t^2/2)^2}(t^2/2)^2
\leq\frac{t^4}{4(1-t^2/2)^2}.
$$
Therefore, putting these bounds together
$$
-t^2-\frac{t^4}{4(1-t^2/2)^2}
\leq u\leq-t^2+\frac{t^4}{3}-\frac{1}{d}\left(-\frac{t^2}{2}+\frac{t^4}{6}\right)^2.
$$
Hence,
$$-t^2-\frac{t^4}{4(1-t^2/2)^2}\leq u\leq-t^2+\frac{t^4}{3}.$$
Finally, using that $e^{u^*}\leq 1$:
$$
t^2-\frac{t^4}{3}-\frac{1}{2}\left(t^2+\frac{t^4}{4(1-t^2/2)^2}\right)^2
\leq -u-\frac{u^2}{2}
\leq 1-\cos^{2d}\left(\frac{t}{\sqrt{d}}\right)
\leq -u
\leq t^2+\frac{t^4}{4(1-t^2/2)^2}.
$$
Thus
\begin{equation*}
{t^2-\frac{t^4}{3}-\frac{1}{2}\left(t^2+\frac{t^4}{4(1-t^2/2)^2}\right)^2\leq 
1-\cos^{2d}\left(\frac{t}{\sqrt{d}}\right)\leq t^2+\frac{t^4}{4(1-t^2/2)^2}}.
\end{equation*}

Similarly, 
%since we are dealing with small values of $t$, 
we have $\cos^{2d-2}(t/\sqrt{d})\geq\cos^{2d}(t/\sqrt{d})$. 
Besides, 
$$\sin^2(t/\sqrt{d})=\left(\frac{t}{\sqrt{d}}+\frac{\sin(t^{***})t^2}{2d}\right)^2,
\quad t^{***}\in[0,t/\sqrt{d}].$$ 
Hence, using the bound for $1-\cos^{2d}$, we have
\begin{multline*}
-d\cos^{2d-2}\left(\frac{t}{\sqrt{d}}\right)\sin^{2}\left(\frac{t}{\sqrt{d}}\right)
\leq -d\cos^{2d}\left(\frac{t}{\sqrt{d}}\right)\left(\frac{t}{\sqrt{d}}+\frac{\sin(t^{***})t^2}{2d}\right)^2\\
\leq d\left(t^2-1+\frac{t^4}{4(1-t^2/2)^2}\right)\left(\frac{t^2}{d}+\frac{\sin(t^{***})t^3}{d^{3/2}}
+\frac{\sin^2(t^{***})t^4}{4d^2}\right)\\
=-t^2+t^3\left[\left[t^2-1+\frac{t^4}{4(1-t^2/2)^2}\right]
\left[\frac{\sin(t^{***})}{d^{1/2}}+\frac{\sin^2(t^{***})t}{4d}\right]+t+\frac{t^3}{4(1-t^2/2)^2}\right]\\
\leq-t^2+t^4+\frac{t^6}{4(1-t^2/2)^2};
\textrm{ if }t<1.
\end{multline*}
Therefore, taking $t<1$
\begin{multline*}
2\pi\cdot v_d(t)\cdot p_d(t)\leq
\frac{\left(t^2+\frac{t^4}{4(1-t^2/2)^2}\right)+\left(-t^2+t^4+\frac{t^6}{4(1-t^2/2)^2}\right)}
{\left(t^2\left[1-\frac{t^2}{3}-\frac{1}{2}\left(t+\frac{t^3}{4(1-t^2/2)^2}\right)^2\right]\right)^{3/2}}\\
=
\frac{t^4\left[1+\frac{1+t^2}{4(1-t^2/2)^2}\right]}
{t^3\left(1-\frac{t^2}{3}-\frac{1}{2}\left(t+\frac{t^3}{4(1-t^2/2)^2}\right)^2\right)^{3/2}}
=\frac{t\left[1+\frac{1+t^2}{4(1-t^2/2)^2}\right]}
{\left(1-\frac{t^2}{3}-\frac{1}{2}\left(t+\frac{t^3}{4(1-t^2/2)^2}\right)^2\right)^{3/2}}.
\end{multline*}
This gives an integrable (at $0$) upper bound.

Besides, Lemma \ref{lemma:limcov} shows that the convergences are uniform in compacts. 

Finally, it rests to obtain an integrable upper bound for large $t$ (but recall that $t\in[0,\sqrt{d}\pi/2]$).
Assume that $t\geq t_0$. 
For $t<a\sqrt{d}$ we can use the bound $\cos^d(t/\sqrt{d})\leq e^{-\alpha t^2/2}$. 
Then, the upper bound is easy to obtain: 
\begin{align*}
2\pi\cdot v_d(t)\cdot p_d(t)
&\leq \frac{1+(1+t^2)e^{-\alpha t^2}}{(1-e^{-\alpha t^2})^{3/2}},\\
|\rho_d(t)|&\leq \frac{e^{-\alpha t^2/2}}{\cos^2(a)}
\frac{1+t^2+e^{-\alpha t^2}}
{1-e^{-\alpha t^2}-t^2e^{-\alpha t^2}}.
\end{align*}

For $t>a\sqrt{d}$ is similar. 
Let $s=t/\sqrt{d}$. 
Let us start with the correlation $\rho_d(t)$. 
Since $\cos(s)\leq \cos(a)$, we have
$$
-d\cos^{2d-2}(s)\sin^2(s)+1-\cos^{2d}(s)
\geq-d\cos^{2d-2}(a)\sin^2(a)+1-\cos^{2d}(a)\to_d1, 
$$
so, the demominator in $\rho_d(t)$ is positive for $d$ large enough. 
%(for instance, for $\cos(a)=0,8$ it suffices to take $d=3$). 
Thus 
\begin{multline}\label{eq:cotacos}
|\rho_d(t)|=\cos^{d-2}(s)
\frac{|1-d\sin^2(s)-\cos^{2d}(s)|}
{1-\cos^{2d}(s)-d\cos^{2d-2}(s)\sin^2(s)}\\
\leq\cos^{d-2}(s)
\frac{|1-d\sin^2(s)-\cos^{2d}(s)|}
{1-d\sin^2(s)-\cos^{2d}(s)}=\cos^{d-2}(s)\leq \cos^{d-2}(a).
\end{multline}

Similarly, since $\cos(s)<\cos(a)$ in this region, we know that the conditional variance and the density (in this region)
are bounded by constants (for instance, $1$ and $(2\pi\sqrt{1-\cos^{2d}(a)})^{-1}$ respectively).

This gives the necessary domination for the cases $\ell\geq 1$ 
and concludes the proof of the Lemma. 
\end{proof}

\end{document}